\documentclass[11pt]{article}
\usepackage{amsmath,amssymb,amsthm,hyperref,comment}
\usepackage{mathtools}
\usepackage{fullpage}
\mathtoolsset{showonlyrefs}
\usepackage[numbers]{natbib}
\usepackage{todonotes}

\newcommand{\N}{{\mathbb N}}
\newcommand{\R}{{\mathbb R}}

\newcommand{\ep}{{\varepsilon}}

\newcommand{\supp}{\,\mathrm{supp}\ }

\newcommand{\sgn}{\mathrm{sgn}}
\newcommand{\E}{\,\mathrm{E}}

\newcommand{\defeq}{\vcentcolon=}

\newcommand{\set}[1]{\left\{ #1 \right\}}
\newcommand{\Set}[2][]{#1\{ #2 #1\}}
\newcommand{\innerprod}[2]{\left\langle #1, #2 \right\rangle}
\newcommand{\size}[1]{\left\lvert #1 \right\rvert}
\newcommand{\abs}[1]{\left\lvert #1 \right\rvert}
\newcommand{\Abs}[2][]{#1\lvert #2 #1\rvert}
\newcommand{\norm}[1]{\left\lVert #1 \right\rVert}
\newcommand{\Norm}[2][]{#1\lVert #2 #1\rVert}
\newcommand{\parens}[1]{\left( #1 \right)}

\newcommand{\bracks}[1]{\left[ #1 \right]}

\newcommand{\ceil}[1]{\left\lceil #1 \right\rceil}

\theoremstyle{plain}
\newtheorem{theorem}{Theorem}[section]
\newtheorem{corollary}[theorem]{Corollary}
\newtheorem{proposition}[theorem]{Proposition}
\newtheorem{lemma}[theorem]{Lemma}

\theoremstyle{definition}
\newtheorem{definition}{Definition}[section]

\theoremstyle{remark}
\newtheorem{remark}{Remark}[section]

\usepackage{color}
\definecolor{red}{rgb}{.8,0,0}
\definecolor{green}{rgb}{0,.7,0}
\definecolor{blue}{rgb}{0,0,.8}
\definecolor{yellow}{rgb}{0.8,0.8,0}

\title{Inverse Expander Mixing for Hypergraphs}
\author{E. Cohen\footnote{School of Mathematics and School of Computer Science, Georgia Institute of Technology; research supported in part by the NSF grants DMS 1101447 and 1407657}, D. Mubayi\footnote{Department of Mathematics, Statistics, and Computer Science, University of Illinois at Chicago; research supported in part by the NSF grants DMS 0969092 and 1300138}, P. Ralli\footnotemark[1], and P. Tetali\footnotemark[1]}
\date{}

\begin{document}

\maketitle

\begin{abstract}
    We formulate and prove inverse mixing lemmas in the settings of simplicial complexes and $k$-uniform hypergraphs. In the  hypergraph setting,  we extend results of Bilu and Linial for graphs. In the simplicial complex setting, our results answer a question of  Parzanchevski et al.
\end{abstract}

\section{Introduction}

Beginning with the seminal work of Thomason \cite{T} and Chung-Graham-Wilson \cite{CGW} the theory of quasirandom graphs has served as an important tool in graph theory and the study of random structures.
Their work showed that many graphs share a collection of common properties, including local properties like subgraph counts, and global properties like edge distribution and eigenvalues.
The connection between these different concepts has proved to be an invaluable tool in graph theory and computer science over the past two decades.
Even the recent theory of graph limits has borrowed many insights from quasirandom graphs \cite{Lov12}.
A fundamental result in this theory is the quantitative relationship between spectral and expansion properties of a graph.
Although early versions of this were proved by Tanner \cite{Tan84} and Alon \cite{Alon86}, the formulation of the result below, which has been termed as the Expander Mixing Lemma for graphs, is perhaps the most popular one.
For a graph $G$, let $\lambda(G)$ be the second largest eigenvalue, in absolute value, of the adjacency matrix $A(G)$.

\begin{theorem}[Expander Mixing Lemma for graphs, \citet{AC}]
    \label{thm:mixing_graph}
    If $G$ is an $r$-regular graph   then for any $S,T\subseteq V(G)$, 
    \begin{equation}
        \abs{E(S,T)-\tfrac{r}{n}\size{S}\size{T}}
        \leq \lambda(G)\sqrt{\size{S}\size{T}},
    \end{equation}
    where $E(S,T)$ is the number of  $(s, t) \in S\times T$ such that $st \in E(G)$.
\end{theorem}

The converse to Theorem~\ref{thm:mixing_graph} is essentially a question about approximating certain quadratic forms on the sphere by points in a cube.
This was achieved by Bilu and Linial \cite{BL} who proved the following converse to the Expander Mixing Lemma.

\begin{theorem}[Inverse Mixing Lemma for graphs, \citet{BL}]
    \label{thm:inverse_mixing_graph}
    If $G$ is an $r$-regular graph such that for every disjoint $S,T\subset V(G)$  
    \begin{equation}
        \abs{E(S,T)-\tfrac{r}{n}\size{S}\size{T}}
        \leq \rho\sqrt{\size{S}\size{T}},
    \end{equation}
    then 
    \begin{equation}
        \lambda(G) = O(\rho\,(\log(r/\rho)+1)).
    \end{equation}
\end{theorem}

\begin{remark}
    This theorem is a converse to the mixing lemma. 
    However, we borrow the language used in \cite{PRT13}, and we refer to this as the Inverse Mixing Lemma.
\end{remark}

Soon after the initial papers on quasirandom graphs, quasirandom properties of other structures were investigated including $k$-uniform hypergraphs for fixed $k \ge 2$ \cite{ChungGraham}. 
It was quickly observed by Chung and Graham and R\"odl that there could be no straightforward generalization of the theory of quasirandom graphs to hypergraphs.  
Over the past two decades, many researchers have extended and generalized parts of the theory and today the situation for hypergraphs is much more satisfactory \cite{Chung12, FW, KNRS, CHPS, LM1, LM2, Towsner13}. 
In particular, the recent work in \cite{LM} defines a hypergraph eigenvalue that originated in the work of Friedman and Wigderson \cite{FW}. 
An analogue of Theorem \ref{thm:mixing_graph} was already proved in \cite{FW}, and this was extended to more general situations in \cite{LM}. 
In the present work, we prove a converse to the mixing lemma for hypergraphs, in other words, a hypergraph version of Theorem \ref{thm:inverse_mixing_graph}.

Another direction in which the theory of quasirandom graphs has developed is to simplicial complexes.  
Indeed, the connection between topological combinatorics and random graphs is proving to be a fruitful and successful area of current research \cite{K}.  
Inspired by the success of Cheeger-type inequalities in graphs---relating isoperimetry (namely, vertex and edge expansion of a graph) to spectral information (eigenvalues of the Laplacian of the graph)---along with the aforementioned expander mixing lemmas in graphs and hypergraphs,  several researchers have initiated developing such connections in the simplicial complex settings. 
Indeed, in recent work Parzanchevski, Rosenthal, and Tessler \cite{PRT13} proved an expander mixing lemma for simplicial complexes (of arbitrary dimension), possessing a complete skeleton. 
(See also an independent development by Steenbergen, Klivans and Mukerjee \cite{SKM12}.) 
In subsequent work Parzanchevski \cite{P13} extended this development by removing the assumption of a complete skeleton and showing that concentration of spectra of the so-called Hodge Laplacian in all dimensions implies combinatorial expansion in any complex. 
An important question left open by Parzanchevski et al was whether a {\em converse} to the expander mixing lemma of theirs held true, much in the spirit of the Bilu-Linial converse \cite{BL} to the mixing lemma of Alon-Chung \cite{AC}. 
Our first main theorem in this work provides an answer to this question, by way of an {\em inverse mixing lemma} for simplicial complexes.



\section{Simplicial Complex Setting}
Our main result for simplicial complexes (Theorem \ref{thm:inverse_mixing_simplicial}) requires some notation and preliminaries which are discussed in the next section. 

\subsection{Notation for simplicial complexes}

For a fuller discussion of notation and preliminaries, refer to \cite{PRT13}. Throughout, let $X$ be a $d$-dimensional simplicial complex with vertex set $V$ of size $n$. Let  $X^i$ denote the set of $i$-cells of $X$, where an $i$-cell consists of $i+1$ vertices so that the simplex defined by these points has dimension $i$.  The simplicial complex $X$ is said to have a \emph{complete skeleton} if $X^i = {V\choose i}$ for each $i = 0, \dots, d-1$.
For any subsets $S_0,\dots,S_d\subseteq V$, we write $F(S_0,\dots,S_d)$ for the number of ordered tuples $(s_0, \ldots, s_d) \in S_0\times \dots \times S_d$ such that $\{s_0, \ldots, s_d\} \in X^d$.

If $i > 0$, an $i$-cell $\set{\sigma_0,\dots,\sigma_i}$ has two orientations given by the orderings of its vertices up to an even permutation.  Denote one orientation by $\sigma = (\sigma_0,\dots,\sigma_i)$ and the other orientation by $\overline{\sigma}$. Let  $X^i_{\pm}$ denote the set of all oriented $i$-cells.   Sometimes, abusing this notation, we also use $\sigma$ to refer to the {\em unoriented} $i$-cell which corresponds to the oriented $i$-cells $\sigma$ and $\overline{\sigma}$.

If $\sigma$ is an oriented $i$-cell and $v\in V-\sigma$, we write $v\sigma$ for the oriented $i+1$-cell $(v,\sigma_0,\dots,\sigma_i)$ and we say $v\sim \sigma$ if and only if $v\sigma\in X^{i+1}$.  For $\sigma\in X^{d-1}$ write $\deg(\sigma) \defeq \size{\{v\in V: v\sim \sigma\}}$ for its degree, and say that $X$ is \emph{$r$-regular} if $\deg(\sigma) = r$ for all $\sigma\in X^{d-1}$.

\begin{definition}
    Let $\Omega^i$ be the vector space of real-valued, skew-symmetric functions on $X^i_{\pm}$, i.e., functions $f:X^i_{\pm}\to\R$ so that $f(\overline{\sigma}) = -f(\sigma)$ for all $\sigma \in X^i_{\pm}$.
\end{definition}

For example, we can think of $\Omega^0$ as the set of vertex weightings, and $\Omega^1$ as the set of flow functions.

\begin{definition}
    Define an inner product on $\Omega^i$ by 
    \begin{align}
        \innerprod{f}{g} \defeq \sum_{\sigma\in X^i} f(\sigma)g(\sigma),
    \end{align}
    noting that $f(\sigma)g(\sigma)=f(\overline{\sigma})g(\overline{\sigma})$, and that we only take one of these terms in the sum.
\end{definition}
With this inner product comes an associated norm $\norm{f} \defeq \sqrt{\innerprod{f}{f}}$ on $\Omega^i$.  Recall that for an operator $M:\Omega^i\to \Omega^i$ (or on any normed vector space) we also have an operator norm, 
\begin{align}
    \norm{M} \defeq \sup_{f\in \Omega^i} \frac{\norm{Mf}}{\norm{f}}.
\end{align}

\begin{definition}
    Define the \emph{boundary operator} $\partial_{d-1}:\Omega^{d-1}\to \Omega^{d-2}$ by
    \begin{align}
        (\partial_{d-1}f)(\tau) \defeq \sum_{v\sim \tau} f(v\tau)
    \end{align}
    and let $Z_{d-1} \defeq \ker\partial_{d-1}$.
\end{definition}  

The boundary of a weight function is the total weight of vertices, while the boundary of a flow is the function that assigns to each vertex the net flow at that vertex.  Correspondingly, $Z_0$ is the set of vertex-weightings with weights summing to 0, while $Z_1$ is the set of conservative flows.





\begin{definition}
    For every $d-2$-cell $\tau$, define linear operators $A_{\tau}, J_\tau:\Omega^{d-1}\to \Omega^{d-1}$ by 
    \begin{align}
        (A_{\tau}f)(\sigma) \defeq \begin{cases}
            \displaystyle\sum_{w\sim v\tau} f(w\tau) &\text{if }\sigma = v\tau\\
            \displaystyle\sum_{w\sim v\tau} f(\overline{w\tau}) &\text{if }\sigma = \overline{v\tau}\\
            0 &\text{if }\tau\not\subset \sigma
        \end{cases}
        &&\text{\quad and \quad}&&
        (J_{\tau}f)(\sigma) \defeq \begin{cases}
            \displaystyle\sum_{w\sim \tau} f(w\tau) &\text{if }\sigma = v\tau\\
            \displaystyle\sum_{w\sim \tau} f(\overline{w\tau}) &\text{if }\sigma = \overline{v\tau}\\
            0 &\text{if }\tau\not\subset \sigma.
        \end{cases}
    \end{align}
    Let $A \defeq \sum_{\tau\in X^{d-2}} A_\tau$ be the adjacency operator, and let $J \defeq \sum_{\tau \in X^{d-2}}J_\tau$.  Denote by $I$ the identity operator on $\Omega^{d-1}$.  
\end{definition}

\begin{remark} \label{rmk:matrices}

We could also view these definitions in a more linear-algebraic light.
For each $\sigma \in X^{d-1}$ we can choose a canonical ``positive'' orientation $\sigma = (\sigma_0, \dots, \sigma_{d-1})$.
We identify $\Omega^{d-1}$ with the vector space $\R^{X_+^{d-1}}$ and operators such as the adjacency operator $A$ above could then be considered as matrices indexed by the canonical orientations of $i$-cells.  With these identifications the inner product and norms defined above correspond to their usual counterparts.

The matrix for the adjacency operator $A_\tau$ would be the signed adjacency matrix of the \emph{graph} induced on $d$-cells and $d-1$-cells containing $\tau$, with $w\tau\sim v\tau$ (i.e, a $\pm 1$ in the coordinate corresponding to the positive orientations of $w\tau$ and $ v\tau$) if and only if $wv\tau\in X^d$.  The signs are determined by the orientations of $w\tau$ and $v\tau$ relative to the canonical positive orientations of those cells (for instance, if the positive orientations are $w\tau$ and $\overline{v\tau}$ then the entry is negative).

Similarly, the matrix for $J_\tau$ would have a $\pm 1$ for any pair of $(d-1)$-cells (not necessarily distinct) both containing $\tau$, with the signs determined in the same fashion.

More precisely, if $\sigma$, $\sigma' \in X^{d-1}$ are positively-oriented cells that differ by exactly one vertex, let $\pi_{\sigma,\sigma'}$ be the unique permutation of $\set{0,\dots,d-1}$ so that $\sigma_{\pi_{\sigma,\sigma'}(i)} = \sigma'_{i}$ whenever $\sigma'_i\in \sigma$.  Then we can calculate the $\sigma,\sigma'$ entry of $A$ and $J$:

\begin{align}
        A_{\sigma,\sigma'} = \begin{cases}
            \sgn(\pi_{\sigma,\sigma'}) &\text{if }\sigma\cup\sigma'\in X^d\\
            0     &\text{otherwise.}
        \end{cases}
    &&\text{\quad and\quad }&&
        J_{\sigma,\sigma'} = \begin{cases}
            \sgn(\pi_{\sigma,\sigma'}) &\text{if }\size{\sigma\cup\sigma'} = d+1,\\
            d &\text{if }\sigma = \sigma'\\
            0     &\text{otherwise.}
        \end{cases}           
    \end{align}

         Note that for ease of analysis positive orientations can be chosen for any \emph{particular} $\tau$ to make all of the signs in $A_\tau$ positive, but this cannot be maintained across all $(d-2)$-cells $\tau$ simultaneously and so the matrix for $A$ must exhibit both signs, regardless of the choice of canonical orientations. 
\end{remark}



For graphs each cell has only one orientation, so we have $X^0 = V$ and $\Omega^0 = \R^V$ and we can think of the usual adjacency matrix of a graph as an operator $A:\R^V\to \R^V$.  Indeed, for $d=1$ the only $(-1)$-cell is the empty set, and so $A = A_\emptyset$ is just the adjacency matrix of the graph, while $J_\emptyset$ is the all-ones matrix.



\begin{definition}
    Finally, define the \emph{degree operator} $D:\Omega^{d-1}\to\Omega^{d-1}$ by 
    \begin{align}
        (Df)(\sigma) \defeq \deg(\sigma) f(\sigma),
    \end{align}
    and define $\Delta^+ \defeq D - A$.
\end{definition}

Note that for $d=1$, $\Delta^+$ is the graph Laplacian.


\subsection{Mixing Lemmas for Simplicial Complexes}

The following is a recent mixing lemma due to Parzanchevski, Rosenthal \& Tessler:

\begin{theorem}[Mixing Lemma for simplicial complexes, \citet{PRT13}]
    \label{thm:mixing_simplicial}
    Let $X$ be a $d$-dimensional complex with a complete skeleton and fix $\alpha\in \R$.  For any disjoint sets $S_0,\dots,S_d\subseteq V$, 
    \begin{align}
        \abs{F(S_0,\dots,S_d)-\tfrac{\alpha}{n} \size{S_0} \dots \size{S_d}} 
        \leq \rho_{\alpha}\sqrt{\size{S_0}\size{S_1}}\size{S_2}\dots \size{S_d},
    \end{align}
    where $\rho_{\alpha} \defeq \norm{(\alpha I-\Delta^+)|_{Z_{d-1}}}$.
\end{theorem}
This is not quite the statement given in the original paper, which concludes with a slightly looser but more symmetric result.  Note that if $X$ is $r$-regular and $\alpha = r$ then $\rho_\alpha$ is the second-largest eigenvalue of $A = r I - \Delta^+$.

The first result we prove is an inverse of the mixing lemma for simplicial complexes.

\begin{theorem}[Inverse Mixing Lemma for simplicial complexes]
    \label{thm:inverse_mixing_simplicial}
    Let $X$ be a $d$-dimensional, $r$-regular simplicial complex with a complete skeleton, and suppose that for every collection of disjoint sets $S_0,\dots,S_d\subseteq V$
    \begin{align}
        \label{eqn:simplicial_rho}
        \abs{F(S_0,\dots,S_d)-\tfrac{r}{n}\size{S_0}\dots \size{S_d}} 
        \leq \rho\sqrt{\size{S_0}\size{S_1}} \size{S_2}\dots \size{S_d}.
    \end{align}
 
    Then 
    \begin{align}
        \norm{A\big|_{Z_{d-1}}} = O(\rho d (\log(r/\rho)+1) + d).
    \end{align}
\end{theorem}

Again, when the complex is regular with a complete skeleton this quantity is the second-largest eigenvalue of $A$.

\begin{remark}
    It is possible to generalize this result by replacing the $r$ in (\ref{eqn:simplicial_rho}) with an arbitrary value $\alpha\in \R$, and by replacing $A$ with $\alpha I - \Delta^+$, as in the statement of \ref{thm:mixing_simplicial}.
\end{remark}

\begin{remark}
    Adding or removing $d-1$-cells does not change the eigenvalues of $A$, but does change $Z_{d-1}$.  The complete skeleton requirement is not necessary to make the statement true of the second-largest eigenvalue of $A$, but \emph{is} necessary to make $\lambda_2 = \norm{A|_{Z_{d-1}}}$.  Observe that $Z_{d-1}$ is basically the space of functions orthogonal to the all-ones function, but the eigenfunction of $A$ corresponding to eigenvalue $r$ has zeroes where there are isolated (i.e. missing) $(d-1)$-cells.
\end{remark}

\begin{remark}
In our proof, we do not use the full strength of the hypothesis.  We will always take $S_2,\dots, S_d$ to be singletons.
\end{remark}




\begin{proof} It is clear from the graph interpretation of $A_\tau$ that the largest eigenvalue of $A_{\tau}$ is $r$, with eigenfunction $f(v\tau) = 1$ (and $f(\overline{v\tau}) = -1$) if $v\sim \tau$ and $f(\sigma) = 0$ if $\tau\not\subset\sigma$.  We will bound $\norm{A_{\tau}-\tfrac{r}{n}J_{\tau}}$, which is an approximation of the second eigenvalue of $A_\tau$.

    First we argue that $J_\tau|_{Z_{d-1}} = 0$.  To see this, consider $f\in Z_{d-1}$, $\tau\in X^{d-2}$, and $\sigma \in X^{d-1}$.  If $\tau\not\subset \sigma$ then $(J_{\tau}f)(\sigma) = 0$.  On the other hand, if $\tau\subset\sigma$ then we can write $\sigma = v\tau$ for some $v\notin \tau$.  In this case
    \begin{align}
        (J_{\tau}f)(\sigma) 
        = \sum_{w\sim \tau} f(w\tau) 
        = (\partial_{d-1}f)(\tau) = 0,
    \end{align}
    so we have $J_{\tau}f\equiv 0$ for every $f\in Z_{d-1}$, or in other words $J_{\tau}|_{Z_{d-1}} = 0$.
  
    \ 
   
    This allows us to say that 
    \begin{align}
        \norm{A |_{Z_{d-1}}} 
        &= \norm{\parens{A-\sum_{\tau\in X^{d-2}}\tfrac{r}{n}J_{\tau}}\Big|_{Z_{d-1}}} \\
        &\leq \norm{A-\sum_{\tau\in X^{d-2}}\tfrac{r}{n}J_{\tau}} \\
        &= \norm{\sum_{\tau\in X^{d-2}}\parens{A_{\tau}-\tfrac{r}{n}J_{\tau}}}.\label{eqn:sum_to_bound}
    \end{align}
  
    What remains is to bound \eqref{eqn:sum_to_bound}. 
    We use a lemma of Bilu \& Linial:
  
    \begin{lemma}[\citet{BL}]
        \label{thm:bilu_linial}
        Let $B$ be a symmetric, real-valued $n\times n$ matrix in which the diagonal entries are all $0$.  Suppose that the $\ell^1$-norm of every row of $B$ is $O(m)$, and also that for any vectors $x,y\in \set{0, 1}^n$ with disjoint support
        \begin{align}\label{eqn:bl}
            \abs{\innerprod{x}{By}} \leq \beta\norm{x}\norm{y}.
        \end{align}
        Then 
        \begin{align}
            \norm{B} = O(\beta(\log(m/\beta)+1)).
        \end{align}
    \end{lemma}  
  
  
  
    We will apply this lemma to 
    \begin{align}
        B = A-\tfrac{r}{n}J + \tfrac{rd}{n}I.
    \end{align} 
 

As mentioned in Remark \ref{rmk:matrices}, we can interpret $B$ as a matrix indexed by positive orientations of elements in $X^{d-1}$.  Combining the calculations of $A$ and $J$ in that remark, we can calculate that for each $\sigma,\sigma'\in X^{d-1}$,
    \begin{align}
        B_{\sigma,\sigma'} = \begin{cases}
            \sgn(\pi_{\sigma,\sigma'})\parens{1-r/n} &\text{if }\sigma\cup\sigma'\in X^d\\
            \sgn(\pi_{\sigma,\sigma'})\parens{-r/n}  &\text{if }\sigma\cup\sigma'\in \binom{V}{d+1}\setminus X^d\\
            0     &\text{otherwise.}
        \end{cases}
    \end{align}
  
    Then we can see that $B$ is symmetric (because $\pi_{\sigma',\sigma} = \pi_{\sigma,\sigma'}^{-1}$), real-valued  and its diagonal entries are 0. Since $X$ is $r$-regular,  the $\ell^1$-norm of each row $\sigma$ in $B$ is 
    \begin{align}
        \sum_{\sigma'\in X^{d-1}}\abs{B_{\sigma, \sigma'}} 
        &=\abs{\sgn(\pi_{\sigma, \sigma'})dr\parens{1-\tfrac{r}{n}} + \sgn(\pi_{\sigma, \sigma'})(n-d-r)d\tfrac{r}{n}} \\
        &= dr\parens{1-\tfrac{r}{n}} + (n-d-r)d\tfrac{r}{n} \leq 2dr.
    \end{align}
    Indeed, there are $(n-d)$ total sets $\eta$ of size $d+1$ containing $\sigma$, and $r$ of those are $d$-cells; each such set $\eta$ contains $d$ other cells $\sigma'$ such that $\sigma\cup\sigma'=\eta$.
  
    Let $x, y: X^{d-1}_{\pm}\to \set{0,\pm 1}$ be functions in $\Omega^{d-1}$ with disjoint support, so that clearly $\innerprod{x}{Iy} = 0$.  For each $\tau\in X^{d-2}$, define $x_{\tau}(\sigma) = x(\sigma)$ if $\tau\subset\sigma$ and $x_{\tau}(\sigma) = 0$ otherwise.  Define $y_{\tau}$ similarly.  Note that each $\sigma\in\supp x$ is in the support of exactly $d$ of the functions $x_\tau$.
    Observe that 
    \begin{align}
        \innerprod{x}{By} 
        &= \innerprod{x}{\parens{\sum_{\tau\in X^{d-2}}A_{\tau}-\tfrac{r}{n}J_{\tau}}y} + \tfrac{rd}{n} \innerprod{x}{Iy} \\
        &= \sum_{\tau\in X^{d-2}}\innerprod{x}{\parens{A_{\tau} - \tfrac{r}{n}J_{\tau}}y} 
        =  \sum_{\tau\in X^{d-2}}\innerprod{x_{\tau}}{A_{\tau} y_{\tau}} - \tfrac{r}{n} \innerprod{x_{\tau}}{J_{\tau}y_{\tau}}.
    \end{align}

    


By definition, for any fixed $\tau = (\tau_2,\dots,\tau_d)$
\begin{align}
    \innerprod{x_\tau}{A_\tau y_\tau} 
    &= \sum_{\sigma\in X^{d-1}} x_\tau(\sigma) (A_\tau y_\tau)(\sigma)
    =\sum_{v\not\in\tau} x(v\tau) (A_\tau y_\tau)(v\tau)\\
    &= \sum_{v\not\in\tau} x(v\tau) \sum_{w\sim v\tau} y_\tau(w\tau)
    = \sum_{v,w} 1_{[vw\tau\in X^d]}x(v\tau)y(w\tau).
\end{align}
Using a similar decomposition for $\innerprod{x_\tau}{J_\tau y_\tau}$ we obtain
\begin{align}
    \innerprod{x_\tau}{(A_\tau - \tfrac{r}{n} J_\tau) y_\tau} 
    = \mathop{\sum_{v\neq w}}_{v,w\not\in\tau} (1_{[vw\tau\in X^d]} - \tfrac{r}{n})x(v\tau)y(w\tau).
\end{align}
We would like to interpret the first half of the sum as a number of edges and the second half as the product of the sizes of some vertex sets, since for any disjoint sets $S_0,S_1$ we have by assumption that
\begin{align}
    \Abs[\Big]{\mathop{\sum_{v\in S_0}}_{w\in S_1} (1_{[vw\tau\in X^d]} - \tfrac{r}{n})}
    &= \abs{F(S_0,S_1,\set{\tau_2},\dots,\set{\tau_d}) - \tfrac{r}{n} \size{S_0}\size{S_1}\size{\set{\tau_2}}\dots\size{\set{\tau_d}}}\\
    &\leq \rho \sqrt{\size{S_0}\size{S_1}}.
\end{align}
However, this differs from what we have above by the signs from $x$ and $y$.  Instead we break the sum apart according to these values, and for each $\eta\in\set{\pm 1}$ we write
\begin{align}
S_0^\eta &= \set{v : x(v\tau) = \eta} & S_1^\eta = \set{w : y(w\tau) = \eta}.
\end{align}
These four sets are pairwise disjoint and $\norm{x_\tau}^2 = \size{S_0^+} + \size{S_0^-}$ and $\norm{y_\tau}^2 = \size{S_1^+} + \size{S_1^-}$, so now we can write
\begin{align}
    \abs{\innerprod{x_\tau}{(A_\tau - \tfrac{r}{n} J_\tau) y_\tau}}
    &= \Abs[\Big]{\sum_{v,w} 1_{vw\tau\in X^d}x(v\tau)y(w\tau)}
    = \Abs[\Big]{\sum_{\eta_0,\eta_1\in \set{\pm 1}} \eta_0\eta_1 \mathop{\sum_{v\in S_0^{\eta_0}}}_{w\in S_1^{\eta_1}} (1_{[vw\tau\in X^d]} - \tfrac{r}{n})}\\
    &\leq \sum_{\eta_0,\eta_1\in \set{\pm 1}} \rho\sqrt{\size{S_0^{\eta_0}} \size{S_1^{\eta_1}}} = \rho \sum_{\eta_0\in\set{\pm1}} \sqrt{\size{S_0^{\eta_0}}} \sum_{\eta_1\in\set{\pm1}} \sqrt{\size{S_1^{\eta_1}}}\\
    &\leq \rho \sqrt{2 \sum_{\eta_0\in\set{\pm1}} \size{S_0^{\eta_0}}}\ \sqrt{2 \sum_{\eta_0\in\set{\pm1}} \size{S_0^{\eta_0}}}
    = 2\rho \norm{x_\tau}\norm{y_\tau}
\end{align}
by Cauchy-Schwarz.
    Summing over all $\tau\in X^{d-2}$ gives that 
    \begin{align}
        \abs{\innerprod{x}{ By}}
        &=\abs{\sum_{\tau\in X^{d-2}}\innerprod{x_{\tau}}{(A_{\tau} - \tfrac{r}{n}J_{\tau})y_{\tau}}} \\
        &\leq \sum_{\tau\in X^{d-2}} \abs{\innerprod{x_{\tau}}{(A_{\tau}-\tfrac{r}{n} J_{\tau})y_{\tau}}}\\
        &\leq \sum_{\tau\in X^{d-2}} 2\rho\sqrt{\size{\supp x_{\tau}}\size{\supp y_{\tau}}}\\
        &\leq 2\rho \sqrt{\sum_{\tau\in X^{d-2}}\size{\supp x_{\tau}}}\ \sqrt{\sum_{\tau\in X^{d-2}}\size{\supp y_{\tau}}} \label{eqn:cauchy_schwarz}\\
        &=2\rho\sqrt{d\size{\supp x}}\ \sqrt{d\size{\supp y}} \\
        &=2\rho d \norm{x}\norm{y},
    \end{align}
    where the inequality in \eqref{eqn:cauchy_schwarz} follows from Cauchy-Schwarz.
  
    Finally we can apply Lemma \ref{thm:bilu_linial} with $m = 2r d$ and $\beta = 2\rho d$ to get
    \begin{align}
        \norm{\parens{\sum_{\tau\in X^{d-2}}A_{\tau}-\tfrac{r}{n}J_{\tau}}+\tfrac{rd}{n}I}
        = O(\rho d (\log(r/\rho)+1)).
    \end{align}  
    Combining the results for each $\tau$ using the triangle inequality gives
    \begin{align}
        \norm{A|_{Z_{d-1}}} 
        \leq \norm{\sum_{\tau\in X^{d-2}}A_{\tau}-\tfrac{r}{n}J_{\tau}}
        = O(\rho d (\log(r/\rho)+1)) + \tfrac{rd}{n} 
        = O(\rho d (\log(r/\rho)+1) + d).
    \end{align}
  
\end{proof}
  
As long as $\emptyset\subsetneq X^d \subsetneq \binom{V}{d+1}$, $\rho\geq \max\set{1-\tfrac{r}{n},\tfrac{r}{n}} \geq 1/2$ (take $S_0, \dots, S_d$ to be singletons corresponding to a subset which is either a $d$-cell or not), so $d = O(\rho d)$ and we can replace the above bound by
\begin{align}
    \norm{A|_{Z_{d-1}}} = O(\rho d (\log(r/\rho)+1)).
\end{align}
The same bound holds trivially for the empty complex (which has $A=0$), while for the complete complex (with $r = n-d$) we can take $S_2, \dots, S_d$ to be singletons and $\size{S_0} = \size{S_1} = \ceil{\frac{n-d}{2}}$, to get
\begin{align}
    \rho &\geq \frac{F(S_0,\dots,S_d) - \frac{r}{n} \size{S_0}\dots\size{S_d}}{\sqrt{\size{S_0}\size{S_1}}\size{S_2}\dots\size{S_d}}
    = \frac{d}{n}\sqrt{\size{S_0}\size{S_1}}
    = \frac{d(n-d)}{2n}
    \geq \frac{1}{4}
\end{align}
when $1\leq d < n$, so this simpler bound holds in general.

\begin{corollary}
    Let $X$ be an $r$-regular, $d$-dimensional complex with a complete skeleton, and suppose that for every collection of disjoint sets $S_0,\dots,S_d\subseteq V$
  
    \begin{align}
        \abs{F(S_0,\dots,S_d)-\tfrac{r}{n}\size{S_0}\dots \size{S_d}} 
        \leq \rho\sqrt{\size{S_0}\size{S_1}} \size{S_2}\dots \size{S_d}.
    \end{align}
 
    Then 
    \begin{align}
        \norm{A\big|_{Z_{d-1}}} = O(\rho d (\log(r/\rho)+1)).
    \end{align}
\end{corollary}

\section{Friedman-Wigderson Hypergraph Setting}

\subsection{Notation for hypergraph eigenvalues}  


The notion of eigenvalues for hypergraphs that we now describe was developed by Friedman \& Wigderson in \cite{FW}.  Further discussion can be found in \cite{LM}.

Throughout, let $H = (V, E(H))$ be a $k$-uniform hypergraph with vertex set $V = \set{v_1, \dots, v_n}$.  We will only consider hypergraphs $H$ with no loops or multiple edges, that is, $E(H)\subseteq \binom{V}{k}$.  The \emph{degree} $\deg(S)=\deg_H(S)$ of a $(k-1)$-set $S$ of vertices in $H$ is the number of edges containing $S$. Say that $H$ is \emph{$r$-regular} if $\deg(S)=r$ for every $(k-1)$-set $S$.

\begin{definition}[Hypergraph adjacency form]
    Let $A = A_H : \prod_{i=1}^k \R^n \to \R$ be the $k$-linear form defined by
    \begin{align}
        A(e_{i_1}, e_{i_2}, \dots, e_{i_k})
        \defeq\begin{cases}
            1 &\text{if }\set{v_{i_1}, v_{i_2}, \dots, v_{i_k}}\in E(H)\\
            0 &\text{otherwise.}
        \end{cases}
    \end{align}
    for all choices of the standard basis vectors $e_{i_1}, e_{i_2}, \dots, e_{i_k}$.
\end{definition}

\begin{definition}  If $V_1,\dots,V_k$ are subsets of $V$, then let
    \begin{align}
        e_H(V_1,\dots,V_k) \defeq \abs{\Set[\big]{\parens{v_1,\dots,v_k}\in V_1\times\dots\times V_k:\set{v_1,\dots,v_k}\in E(H)}}.
    \end{align}
\end{definition}
    
As with the adjacency form we will suppress the subscript $H$ when the hypergraph is clear from context.  If $V_1,\dots, V_k$ are pairwise disjoint, this is the number of edges that intersect each $V_i$ in exactly one vertex.  Alternatively, if we take $x^i$ to be the indicator vector of $V_i$ then we could equivalently define $e(V_1, \dots, V_k) = A(x^1,\dots,x^k)$.

Let $J$ denote the $k$-linear form with $J(e_{i_1}, e_{i_2}, \dots, e_{i_k}) = 1$ for all choices of standard basis vectors $e_{i_1}, e_{i_2}, \dots, e_{i_k}$.  Let $K = (V, \binom{V}{k})$ denote the complete $k$-uniform hypergraph on vertex set $V$ (with corresponding adjacency form $A_K$ which evaluates to 1 on any \emph{distinct} standard basis vectors.)


\begin{definition}
    If $\displaystyle \phi: \prod_{i=1}^k \R^n \to \R$ is a $k$-linear form, we define the \emph{spectral norm} of $\phi$ to be 
    \begin{align}
        \norm{\phi} \defeq \sup_{x_i\in \R^n,\ x_i\neq 0} \frac{\abs{\phi(x_1, \dots, x_k)}}{\norm{x_1} \dots \norm{x_k}}.
    \end{align}
\end{definition}

In the case where $\phi$ is symmetric, as shown in ~\cite{FW} we in fact have that 
\begin{align}
    \norm{\phi} = \sup_{x\in \R^n,\ x\neq 0} \frac{\abs{\phi(x, \dots, x)}}{\norm{x}^k}.
\end{align}

Observe that both $A$ and $J$ are symmetric.

Recall that the first (largest) eigenvalue of a graph can be defined as the operator norm of its adjacency matrix, $\norm{A_G}$, and if the graph is $r$-regular then the second-largest eigenvalue is $\norm{A_G - \tfrac{r}{n} J}$.  This motivates a definition of the second eigenvalue for \emph{hypergraphs} given by Friedman and Wigderson in \cite{FW}: if $H$ is $r$-regular, they define the second eigenvalue to be
    \begin{align}
        \lambda_2(H) &\defeq \norm{A-\frac{k!\size{E(H)}}{n^k}J}\\
        &= \norm{A-\tfrac{r}{n}J}.      \label{eqn:lambda2}
    \end{align}
    For any $H$ (not necessarily $r$-regular), the quantity in \eqref{eqn:lambda2} is called the \emph{second eigenvalue of $H$ with respect to $r$-regularity}.

For technical reasons, we will take a slightly different definition for the second eigenvalue.


\begin{definition}    
    For any $\alpha\in \R$, the \emph{second eigenvalue of $H$ with respect to $\alpha$-density} is 
    \begin{align}
        \lambda_{2,\alpha}(H) \defeq \norm{A-\alpha A_K} = \sup_{x\in \R^n,\ x\neq 0} \frac{\abs{A(x,\dots,x) - \alpha A_K(x,\dots,x)}}{\norm{x}^k}.
    \end{align}
\end{definition}

We also define a parallel parameter measuring the combinatorial expansion of a hypergraph.
\begin{definition}  For a $k$-uniform hypergraph $H$, define for each $\alpha \geq 0$
    \begin{align}
        \rho_\alpha(H) \defeq \max_{V_1,\dots,V_k}\frac{\Abs[\big]{\size{e(V_1,\dots,V_k)}-\alpha\size{V_1}\dots \size{V_k}}}{\sqrt{\size{V_1}\dots \size{V_k}}},
    \end{align}
    where the maximum is taken over all tuples $V_1,\dots,V_k$ of pairwise disjoint nonempty subsets of $V$.
\end{definition}
    
\begin{remark}
    Our aim is to bound the second eigenvalue in terms of $\rho=\rho_{\alpha}(H)$.
    Unfortunately, we find that independently of $\rho$, $\lambda_2 = \Omega(rn^{k-2})$ with high probability for random hypergraphs with edge density $r/n$, making an inverse mixing lemma for $\lambda_2$ impossible.  
    Our new definition $\lambda_{2,\alpha}$ allows us to avoid this problem.
    A further discussion of this problem is found in Section \ref{sec:counterexamples}.
\end{remark}

It is natural in our definition to choose $\alpha = \size{E(H)}/\size{E(K)}$, i.e., the edge-density of $H$.  However, to more closely parallel the Friedman-Wigderson definition one can choose $\alpha = \frac{r}{n}$.  For now we will proceed without specifying a fixed value for $\alpha$.

\begin{remark}
    Even if one fixes $\alpha$ as suggested above to be the edge density of $H$, our definition does not quite agree with the usual definition of graph eigenvalues in the case of $r$-regular graphs ($k=2$).  In particular, where $\lambda(G) = \Norm{A - \tfrac{r}{n} J}$ we use $\lambda_{2,\alpha}(G) = \Norm{A - \tfrac{r}{n-1} A_K}$.  However, it is easy to see that the two values never differ by more than $\frac{r}{n-1}\Norm{I - \tfrac{1}{n}J} = \tfrac{r}{n-1} \leq 1$.
\end{remark}

The following simple upper bound will come in handy in later analysis.
\begin{proposition}\label{thm:rho_sanity_bound}
    For any $k$-uniform hypergraph $H$ with maximum degree $r$, 
    \begin{align}
        \rho_\alpha(H) \leq (r+\alpha n) n^{(k-2)/2}.
    \end{align}
\end{proposition}
\begin{proof}
    Working directly from the definition, we have
    \begin{align}
        \rho_\alpha(H) 
        &= \max_{V_1,\dots,V_k} \frac{\Abs[\big]{e(V_1,\dots, V_k) - \alpha\size{V_1}\dots\size{V_k}}}{\sqrt{\size{V_1}\dots\size{V_k}}}\\
        &\leq \max_{V_1,\dots,V_k} \frac{e(V_1,\dots, V_k) + \alpha\size{V_1}\dots\size{V_k}}{\sqrt{\size{V_1}\dots\size{V_k}}}\\
        &\leq \max_{\size{V_1}\geq \dots \geq \size{V_k}} \frac{r\size{V_2}\dots\size{V_k} - \alpha\size{V_1}\dots\size{V_k}}{\sqrt{\size{V_1}\dots\size{V_k}}}\\
        &= \max_{\size{V_1}\geq \dots \geq \size{V_k}} (r + \alpha \size{V_1})\frac{\sqrt{\size{V_2}\dots\size{V_k}}}{\sqrt{\size{V_1}}}\\
        &\leq (r+\alpha n) n^{(k-2)/2}.
    \end{align}
\end{proof}


\subsection{Hypergraph Mixing Lemmas}
The following hypergraph mixing result is given in \cite{FW}. 

\begin{theorem}[Mixing Lemma for hypergraphs, \citet{FW}]
    \label{thm:mixing_hypergraph}
    Let $H$ be a $k$-uniform hypergraph.  For any choice of subsets $V_1,\dots,V_k\subset V(H)$ of vertices, 
    \begin{align}
        \abs{\size{e(V_1,\dots,V_k)}-\frac{k!\size{E(H)}}{n^k}\size{V_1}\dots \size{V_k}}
        \leq \lambda_2(H)\sqrt{\size{V_1}\dots \size{V_k}}.
    \end{align}
\end{theorem}



Before stating and proving a converse to Theorem~\ref{thm:mixing_hypergraph} above, we  mention the mixing result using our definition of the second eigenvalue $\lambda_{2,\alpha}$, with respect to density $\alpha$.

\begin{theorem}[Mixing Lemma for hypergraphs]
\label{thm:mixing_hypergraph_density}
 Let $H$ be a $k$-uniform hypergraph.  For any choice of subsets $V_1,\dots,V_k\subset V(H)$ of vertices, 
    \begin{align}
        \abs{e(V_1,\dots,V_k)-\alpha e_K(V_1,\dots,V_k)}
        \leq \lambda_{2,\alpha}(H)\sqrt{\size{V_1}\dots \size{V_k}}.
    \end{align}
\end{theorem}

\begin{proof}
Let $V_1,\dots,V_k\subset V(H)$.  If any $V_i$ is empty, it is clear that the inequality holds; we may assume that each $V_i$ is nonempty.  For $1\leq i\leq k$ let $x^i\in \set{0,1}^n$ be the indicator vector of $V_i$. Then

\begin{align}
    \frac{\abs{e(V_1,\dots,V_k)-\alpha e_K(V_1,\dots,V_k)}}{\sqrt{\size{V_1}\dots \size{V_k}}}
    &=\frac{\abs{A(x^1,\dots ,x^k)-\alpha A_K(x^1,\dots, x^k)}}{\prod_{i=1}^k\norm{x^i}} \\
    &\leq \norm{A-\alpha A_K} = \lambda_{2,\alpha}(H)
\end{align}
as desired.

\end{proof}


We now prove the main theorem of this section -- a converse to the above Theorem \ref{thm:mixing_hypergraph_density}:

\begin{theorem}[Inverse Mixing Lemma for hypergraphs]
    \label{thm:inverse_mixing_hypergraph}
    If $H$ is a $k$-uniform hypergraph with maximum codegree $r$ and $\rho = \rho_\alpha(H)$ 
    then 
    \begin{align}
        \lambda_{2,\alpha}(H) = O\parens{\rho\,(\log^{k-1}((r+\alpha n) n^{k-2}/\rho)+1)}.
    \end{align}
\end{theorem}

\begin{remark} We have left this result in what is perhaps not its simplest form, in order to show the difference between the cases $k=2$ and $k\geq 3$.  In the case where $k=2$ and $\alpha=\Theta(r / n)$ the dependence on $n$ disappears and this simplifies to the classic result $\lambda_{2,\alpha} = O(\rho(\log(r/\rho) + 1))$ for graphs.  For larger (but still constant) uniformity, we can still simplify the result to $\lambda_{2,\alpha} = O(\rho\,(\log^{k-1}((r+\alpha n) n/\rho) + 1))$.
\end{remark}



\newcommand{\bound}{b}
We prove the theorem through a series of lemmas.  First we show that the partite expansion condition suffices to give expansion for any (not necessarily disjoint) sets of vertices.  Throughout, $\bound$ represents a constant independent of $x$ (but which may depend on $k$, $n$, $r$, $\alpha$, $\rho$ or anything else).
\begin{lemma}
    \label{lem:expansion_equivalence}
    Let $H$ be a $k$-uniform hypergraph on $n$ vertices with adjacency form $A$, and suppose that 
    \begin{align}
        \abs{A(x^1,\dots,x^k) - \alpha  A_K(x^1,\dots,x^k)} 
        \leq \rho \prod_{i=1}^k \norm{x^i}
    \end{align}
    for every choice of pairwise orthogonal vectors $x^1,\dots,x^k \in \set{0,1}^n$.  Then
    \begin{align}
        \abs{A(x^1,\dots,x^k) - \alpha  A_K(x^1,\dots,x^k)} 
        \leq \rho k^{k/2} \prod_{i=1}^k \norm{x^i}
    \end{align}
    for every choice of (not necessarily orthogonal) vectors $x^1,\dots,x^k \in \set{0,1}^n$.
\end{lemma}

\begin{proof}
    Let $V_1,\dots,V_k \subseteq [n]$ be any sets of vertices. Consider an ordered partition $\mathcal{P} = P_1\cup \dots \cup P_k$ of $[n]$ into $k$ nonempty parts.  Then
    \begin{align}
        e(V_1,\dots,V_k) 
        &= \frac{1}{k^{n-k}}\sum_{\mathcal P} e(P_1\cap V_1,\dots,P_k\cap V_k),
    \end{align}
    as every ordered edge $(v_1,\dots,v_k)$ shows up in the sum once for each partition $\mathcal{P}$ with $v_j\in P_j$ for every $j$, and there are $(n-k)^k$ such partitions (the remaining $n-k$ elements can be partitioned in any way among the $k$ sets).  Similarly, replacing $H$ with the complete hypergraph gives
    \begin{align}
        e_K(V_1,\dots,V_k)
        &= \frac{1}{k^{n-k}}\sum_{\mathcal P}  e_K(P_1\cap V_1,\dots,P_k\cap V_k)
        = \frac{1}{k^{n-k}}\sum_{\mathcal P} \prod_i \size{P_i\cap V_i}.
    \end{align}
    For a fixed partition the subsets $P_i\cap V_i$ are disjoint, so by hypothesis we have
    \begin{align}
        \abs{e(P_1\cap V_1,\dots,P_k\cap V_k) - \alpha  e_K(P_1\cap V_1,\dots,P_k\cap V_k)}
        &\leq \rho \sqrt{\prod_i \size{P_i\cap V_i}}.
    \end{align}
    Then
    \begin{align}
        \MoveEqLeft\abs{e(V_1,\dots,V_k) - \alpha e_K(V_1,\dots,V_k)}\\
        &\leq \frac{1}{k^{n-k}}\sum_{\mathcal P} \abs{e(P_1\cap V_1,\dots,P_k\cap V_k) - \alpha e_K(P_1\cap V_1,\dots,P_k\cap V_k)}\\
        &\leq \frac{1}{k^{n-k}}\sum_{\mathcal P} \rho \sqrt{\prod_i \size{P_i\cap V_i}}\\
        &= \frac{\rho k!S(n,k)}{k^{n-k}} \sum_{\mathcal{P}} \frac{1}{k! S(n,k)} \sqrt{\prod_i \size{P_i\cap V_i}}\\
        &\leq \frac{\rho k! S(n,k)}{k^{n-k}} \sqrt{\frac{1}{k! S(n,k)} \sum_{\mathcal{P}} \prod_i \size{P_i\cap V_i}}\label{eqn:concavity}\\
        &=\frac{\rho k! S(n,k)}{k^{n-k}} \sqrt{\frac{k^{n-k} e_K(V_1,\dots,V_k)}{k! S(n,k)}}\\
        &\leq \rho \sqrt{\frac{k! S(n,k)}{k^{n-k}}} \prod_i \sqrt{\size{V_i}} \leq \rho k^{k/2} \prod_i \sqrt{\size{V_i}}.
    \end{align}

    Here $k!S(n,k) \leq k^n$ is the number of ordered partitions of $[n]$ into $k$ nonempty sets (the number of terms in the sum over all choices of $\mathcal P$), and the inequality in \eqref{eqn:concavity} follows by concavity of square root.  
    
    The final result follows immediately, noting that if $x^1,\dots,x^k$ are the indicator vectors for $V_1,\dots,V_k$ then $e(V_1,\dots,V_k) = A(x^1,\dots,x^k)$ and $\size{V_i} = \norm{x^i}^2$.
    
\end{proof}

The main part of the work goes towards proving a hypergraph version of Lemma \ref{thm:bilu_linial}.  We go through several steps to show that if the expansion bound holds for $\set{0,1}$ vectors then a somewhat relaxed bound holds for all real vectors.

\begin{lemma}
    \label{lem:zero_pm_one}
    Suppose $B$ is a $k$-linear form such that 
    \begin{align} 
        \abs{B(x^1,\dots,x^k)} \leq \bound \prod_{i=1}^k \norm{x^i}
    \end{align}
    for every $x^1,\dots,x^k\in\set{0,1}^n$.  Then 
    \begin{align} 
        \abs{B(x^1,\dots,x^k)} \leq 2^{k/2} \bound \prod_{i=1}^k \norm{x^i}
    \end{align}
    for every $x^1,\dots,x^k \in \set{0,\pm 1}^n$.
\end{lemma}
\begin{proof}
    Let $x^1,\dots,x^k \in \set{0,\pm 1}^n$, and decompose $x^i = x^i_+ - x^i_-$ so that $x^i_\pm\in \set{0,1}^n$ and $\supp x^i = \supp x^i_+\cup \supp x^i_-$.  
    Then
    \begin{align} 
        \abs{B(x^1,\dots,x^k)} 
        &= \abs{B(x^1_+ - x^1_-, \dots, x^k_+ - x^k_-)}\\
        &\leq \sum_{\eta \in \set{\pm}^k} \abs{B(x^1_{\eta_1},\dots,x^k_{\eta_k})}
         \leq \sum_{\eta \in \set{\pm}^k} \bound \prod_{i=1}^k \norm{x^i_{\eta_i}}\\
        &= \bound \prod_{i=1}^k \parens{\norm{x^i_+} + \norm{x^i_-}}
         \leq \bound \prod_{i=1}^k \sqrt{2} \norm{x^i}.
    \end{align}
\end{proof}

\begin{lemma}
    \label{lem:powers_of_two}
    Suppose $B$ is a symmetric $k$-linear form satisfying
    \begin{align}
    \label{row-sum}
        \sum_{j=1}^n \abs{B(e_{i_1},\dots,e_{i_{k-1}}, e_j)} \leq m
    \end{align}
    for every $(i_1,\dots,i_{k-1})\in [n]^{k-1}$ and 
    \begin{align} 
        \abs{B(x^1,\dots,x^k)} \leq \bound \prod_{i=1}^k \norm{x^i}
    \end{align}
    for every $x^1,\dots,x^k \in \set{0,\pm 1}^n$.  Let $a \geq b/(m n^{(k-2)/2})$. Then
    \begin{align} 
        \abs{B(x,\dots,x)} &\leq \bound \parens{\lg^{k-1} \parens{\frac{a^2 m^2 n^{k-2}}{\bound^2}} + \frac{k^2}{a}}\norm{x}^k
    \end{align}
    for every $x \in \set{0,\pm 2^{-\ell}: \ell\in\N}^n$.
\end{lemma}

\begin{proof}
	Let $x\in\set{0,\pm 2^{-\ell}:\ell\in \N}$ and write $x=\sum_{i\in \N} 2^{-i} x^i$ with $x^i\in \set{0, \pm 1}^n$ (the $x^i$ have pairwise disjoint support and are hence orthogonal).  Define $s_i=\size{\supp x^i} = \norm{x^i}^2$ so that 
	\begin{align}
		\norm{x}_1 &= \sum_{i \in \N} 2^{-i} s_i &\text{and}&&
		\norm{x}_2^2 &= \sum_{i \in \N} 2^{-2i} s_i.
	\end{align}
	Note that all sums have only finitely many nonzero terms.  We are interested in bounding
	\begin{align}
		\abs{B(x,\dots, x)}
		&\leq \sum_{i\in \N^k} \parens{\prod_{j=1}^k 2^{-i_j}} \abs{B(x^{i_1}, \ldots, x^{i_k})}.
	\end{align}

We split this sum into two parts, bounding separately the sums over the index sets
	\begin{align}
		P &= \set{i\in \N^k: \max_j\, i_j - \min_j\, i_j < \gamma} &\text{and} &&
		Q &= \N^k \setminus P
	\end{align}
for some $\gamma \geq 0$ to be determined later.  For the sum over $i\in P$ we have
	\begin{align}
		\sum_{i\in P} \parens{\prod_{j=1}^k 2^{-i_j}} \abs{B(x^{i_1}, \ldots, x^{i_k})}
		&\leq \sum_{i\in P} \parens{\prod_{j=1}^k 2^{-i_j}} \bound \prod_{j=1}^k \sqrt{s_{i_j}} \\
		&= \bound \sum_{i\in P} \parens{\prod_j \parens{2^{-2i_j} s_{i_j}}^{k/2}}^{1/k} \\
		&\leq \frac{\bound}{k} \sum_{i\in P}\sum_j \parens{2^{-2i_j}s_{i_j}}^{k/2},
\intertext{where the final step uses the AM-GM inequality. Each $\ell\in \N$ appears at most $k(2\gamma)^{k-1}$ times in elements of $P$ (as each time $\ell$ appears in some position the remaining $k-1$ terms must all be between $\ell-\gamma$ and $\ell+\gamma$), so}
		\frac{\bound}{k} \sum_{i\in P} \sum_j \parens{2^{-2i_j}s_{i_j}}^{k/2}
		&\leq \bound\, (2\gamma)^{k-1} \sum_{\ell\in \N} \parens{2^{-2\ell}s_{\ell}}^{k/2}\\
		&\leq  \bound\, (2\gamma)^{k-1} \parens{\sum_\ell 2^{-2\ell} s_\ell}^{k/2}\\
		&=  \bound\, (2\gamma)^{k-1} \norm{x}^k
	\end{align}
	where we have used that $\sum_i a_i^{k/2} \leq (\sum_i a_i)^{k/2}$ for nonnegative $a_i$ and $k\geq 2$.
  
	Now we focus on bounding the sum over $i\in Q$.  For each $i\in Q$ we move $\min i$ to $i_1$ and $\max i$ to $i_k$.  Such a reordered index vector corresponds to at most $k^2$ original vectors, so we have
	\begin{align}
		\sum_{i\in Q} \parens{\prod_{j=1}^k 2^{-i_j}} \abs{B(x^{i_1},\dots,x^{i_k})}
		&\leq k^2\sum_{i\in \N^{k-1}} \sum_{i_k\geq i_1 + \gamma} \parens{\prod_{j=1}^k 2^{-i_j}} \abs{B(x^{i_1},\dots,x^{i_k})} \\
		&\leq k^2 \sum_{i\in \N^{k-1}} 2^{-2i_1-\gamma}\parens{\prod_{j=2}^{k-1} 2^{-i_j}} \sum_{i_k\in \N} \abs{B(x^{i_1},\dots,x^{i_k})}. \label{eqn:qbound}
	\end{align}
		
	For fixed $i_1,\dots,i_{k-1}$,
	\begin{align}
		\sum_{i_k\in \N} \abs{B(x^{i_1},\dots,x^{i_k})}
		&\leq \sum_{i_k\in \N}\ \sum_{\ell_1 \in \supp\! x^{i_1}}\!\!\cdots\!\! \sum_{\ell_{k} \in \supp\! x^{i_{k}}} \abs{B(e_{\ell_1},\dots,e_{\ell_{k}})}\\
		&= \sum_{\ell_1 \in \supp\! x^{i_1}}\!\!\cdots\!\! \sum_{\ell_{k-1} \in \supp\! x^{i_{k-1}}}\ \sum_{\ell_k\in [n]} \abs{B(e_{\ell_1},\dots,e_{\ell_{k}})}\\
		&\leq \sum_{\ell_1 \in \supp\! x^{i_1}}\!\!\cdots\!\! \sum_{\ell_{k-1} \in \supp\! x^{i_{k-1}}} m\\
		&= m\prod_{j=1}^{k-1}s_{i_j},
	\end{align} 
	where the last inequality is due to hypothesis \eqref{row-sum}.  Plugging this into the bound \eqref{eqn:qbound} above gives
		
		\begin{align}
		\sum_{i\in Q} \parens{\prod_{j=1}^k 2^{-i_j}} \abs{B(x^{i_1},\dots,x^{i_k})}
		&\leq k^2 \sum_{i\in \N^{k-1}} 2^{-2i_1-\gamma}\parens{\prod_{j=2}^{k-1} 2^{-i_j}} m\prod_{j=1}^{k-1} s_{i_j} \label{degree_bound}\\
		&= k^2 2^{-\gamma} m \sum_{i_1\in \N} 2^{-2i_1} s_{i_1} \sum_{i_2,\dots,i_{k-1}} \prod_{j=2}^{k-1} 2^{-i_j}s_{i_j}\\
		&= k^2 m 2^{-\gamma} \norm{x}^2 \norm{x}_1^{k-2}\\
		&\leq k^2 m n^{(k-2)/2} 2^{-\gamma} \norm{x}^k,
	\end{align}
	
    using the fact that $\norm{x}_1 \leq \sqrt{n} \norm{x}$ (by Cauchy-Schwarz). Putting everything together, we have 
	\begin{align}
		\abs{B(x,\ldots, x)}/\norm{x}^k
		\leq \bound\, (2\gamma)^{k-1} + k^2 m n^{(k-2)/2} 2^{-\gamma}.
	\end{align}
    Finally, set $\gamma = \lg (a m n^{(k-2)/2}/ \bound)$ (which is non-negative by the restriction on $a$) to get
	\begin{align}
	  \abs{B(x,\dots,x)}/\norm{x}^k 
	  &\leq \bound\, (\lg^{k-1} (a^2 m^2 n^{k-2} / \bound^2) + k^2 / a)
	\end{align}
	as desired. 
\end{proof}

\begin{lemma}
    \label{lem:continuum}
    Suppose $B$ is a $k$-linear form such that
    \begin{align} 
        \abs{B(x,\dots, x)} \leq \bound \norm{x}^k
    \end{align}
    for every $x \in \set{0,\pm 2^{-\ell}: \ell\in\N}^n$, and $B(e_{i_1},\dots,e_{i_k}) = 0$ whenever $i_1,\dots,i_k$ are not all distinct.  Then $\norm{B} \leq 2^k \bound$.
\end{lemma}
\begin{proof}
    Let $x\in\R^n$ be a vector which maximizes $\abs{B(x,\dots,x)}/\norm{x}^k = \norm{B}$. Without loss of generality, scale $x$ so that $\abs{x_i} \leq 1/2$ for all $i \in [n]$.

    Choose a random vector $z \in \set{0,\pm 2^{-\ell}: \ell\in \N}^n$ by picking each coordinate $z_i$ independently as follows:

    \begin{quote}
        If $x_i = 0$ then $z_i = 0$.  Otherwise, write $\abs{x_i} = 2^{\ell_i} (1+\ep_i)$ for some integer $\ell_i$ and some value of $\ep_i \in [0,1)$.  Let $z_i = \text{sign}(x_i) 2^{\ell_i}$ with probability $1-\ep_i$ and $z_i = \text{sign}(x_i) 2^{\ell_i+1}$ with probability $\ep_i$.
    \end{quote}

    We can see that $\E[z_i] = x_i$ for all $i\in [n]$ 
    and
    \begin{align}
        \E[B(z,\dots, z)] 
        &= \sum_{i\in[n]^k} \E[B(z_{i_1} e_{i_1},\dots,z_{i_k}e_{i_k})] = \sum_{i\in[n]^k} \parens{\prod_{j=1}^k \E[z_{i_j}]}B(e_{i_1},\dots,e_{i_k})\\
        &= \sum_{i\in[n]^k} \parens{\prod_{j=1}^k x_{i_j}}B(e_{i_1},\dots,e_{i_k})
        = B(x,\dots,x).
    \end{align}
    
    Thus there is a vector $z$ for which $\abs{B(z, \dots, z)}\geq \abs{B(x, \dots, x)}$.  Observe that by construction $\norm{z} \leq 2\norm{x}$, so
    \begin{align}
        \abs{B(x,\dots,x)}
        &\leq \abs{B(z,\dots,z)} 
        \leq b \norm{z}^k
        \leq 2^k \bound \norm{x}^k.
    \end{align} 
    Consequently, $\norm{B} = \abs{B(x,\dots,x)}/\norm{x}^k \leq 2^k \bound$.
\end{proof}
\ 

Finally, we put all of these lemmas together to prove the theorem.

\begin{proof}[Proof of Theorem \ref{thm:inverse_mixing_hypergraph}]
    Suppose $H$ is a $k$-uniform hypergraph on $n$ vertices with maximum degree $r$ satisfying
    \begin{align}
        \abs{e(V_1,\dots,V_k) - \alpha e_K(V_1,\dots,V_k)} 
        \leq \rho \sqrt{\prod_{i=1}^k \size{V_i}}
    \end{align}
    for every choice of disjoint sets $V_1,\dots,V_k\subseteq V(H)$.  By Lemma \ref{lem:expansion_equivalence}, the adjacency form $A$ in fact satisfies 
    \begin{align}
        \abs{A(x^1,\dots,x^k) - \alpha A_K(x^1,\dots,x^k)} 
        \leq \rho k^{k/2} \prod_{i=1}^k \norm{x^i}
    \end{align}
    for every $x^1,\dots,x^k\in\set{0,1}^n$.  Taking $B = A - \alpha A_K$ in Lemma \ref{lem:zero_pm_one} gives that
    \begin{align}
        \abs{B(x^1,\dots,x^k)} \leq \rho (2k)^{k/2} \prod_{i=1}^k \norm{x^i}
    \end{align}
    for all $x^1,\dots,x^k\in\set{0,\pm 1}^n$.  Since for any fixed $i_1,\dots,i_{k-1}$
    \begin{align}
        \sum_{j=1}^n \abs{B(e_{i_1}, \dots,e_{i_{k-1}},e_j)}
        &\leq \sum_{j=1}^n \abs{A(e_{i_1}, \dots,e_{i_{k-1}},e_j)}+\alpha\sum_{j=1}^n \abs{A_K(e_{i_1}, \dots,e_{i_{k-1}},e_j)}\\
        &\leq r+\alpha n, 
    \end{align}
    we can use $m = r + \alpha n$, $b = \rho\, (2k)^{k/2}$ and $a = (2k)^{k/2}$ in Lemma \ref{lem:powers_of_two} (using Proposition \ref{thm:rho_sanity_bound} to guarantee the constraint on $a$) to find that
    \begin{align}
        \abs{B(x,\dots,x)} \leq \rho\, (2k)^{k/2} \parens{\log^{k-1}\parens{\frac{(d+\alpha n)^2}{\rho^2} n^{k-2}} + k^2 (2k)^{-k/2}} \norm{x}^k
    \end{align}
    for every $x\in\set{0, \pm 2^{-\ell}: \ell\in\N}^n$.  Finally, by Lemma \ref{lem:continuum} we find that
    \begin{align}
        \lambda_{2,\alpha}(H) 
        = \norm{B} 
        &\leq 2^{3k/2} k^{k/2} \rho \parens{\log^{k-1}\parens{\frac{(d+\alpha n)^2}{\rho^2} n^{k-2}} + k^2 (2k)^{-k/2}}\\
        &= \rho\, O(\log^{k-1} ((r+\alpha n)n^{k-2}/\rho) + 1).
    \end{align}
\end{proof}


\subsection{Comparison with the Friedman-Wigderson definition of $\lambda_2$}\label{sec:counterexamples}

In this section we prove an inverse mixing lemma for the Freidman-Wigderson definition of the second eigenvalue.  We will see that this result, while tight, is not as useful as theorem \ref{thm:inverse_mixing_hypergraph}, and we briefly discuss the reason for this.

First, we define a useful $k$-linear form and evaluate its norm.

\begin{definition}
   Let $D = J - A_K$ denote the $k$-linear form with $D(e_{i_1},\dots,e_{i_k}) = 1$ if and only if the indices $i_j$ are not all distinct (and 0 otherwise).

\end{definition}

\begin{proposition} \label{D_bound}
$\norm{D} = \Theta(n^{(k-2)/2})$.
\end{proposition}
\begin{proof}
    First of all note that
    \begin{align}
        \norm{D}
        &\geq \Abs{D(\vec 1, \dots, \vec 1)}/\Norm{\vec 1}^k\\
        &= \frac{n^k - n!/(n-k)!}{n^{k/2}} = \Omega(n^{(k-2)/2}).
    \end{align}
    
    On the other hand, for any $x\in \R^n$
    \begin{align}
        \abs{D(x,\dots,x)} 
        &\leq \sum_{i\in [n]^k} \parens{\prod_{j=1}^k \abs{x_{i_j}}} \abs{D(e_{i_1},\dots,e_{i_k})}\\
        &= \mathop{\sum_{i\in [n]^k}}_{i_j\text{ not all distinct}}\prod_{j=1}^k \abs{x_{i_j}}\\
        &\leq k^2 {\sum_{i\in [n]^{k-1}}} \abs{x_{i_1}} \prod_{j=1}^{k-1} \abs{x_{i_j}}\\
        &= k^2 \sum_{i_1=1}^n \abs{x_{i_1}}^2 \prod_{j=2}^{k-1} \sum_{i_j = 1}^n \abs{x_{i_j}}\\
        &= k^2 \norm{x}_2^2 \norm{x}_1^{k-2}\\
        &\leq k^2 n^{(k-2)/2} \norm{x}^k,
    \end{align}
    as desired.
\end{proof}


\begin{theorem}
\label{inverse_mixing_fw}

 Let $H$ be a $k$-uniform hypergraph with maximum degree $r$, and suppose that for every choice of disjoint sets $V_1,\dots,V_k\subset V(H)$, 
    \begin{align}
        \abs{e(V_1,\dots,V_k)-\frac{r}{n} \prod_{i} \size{V_i}}
        \leq \rho\sqrt{\prod_{i = 1}^k \size{V_i}}\,.
    \end{align}  
    Then 
    \begin{align}
        \lambda_2(H) = \Theta(rn^{(k-4)/2}) \pm O\parens{(\log^{k-1}(rn^{k-2}/\rho)+1)\rho}.
    \end{align}
    
\end{theorem}
\begin{proof}  Set $\alpha = \tfrac{r}{n}$, we have that $r = \Theta(\alpha n)$.  Observe that if $V_1,\dots,V_k$ are disjoint, then 
\begin{align}
    \abs{e(V_1,\dots,V_k)-\tfrac{r}{n}\, e_K(V_1,\dots,V_k)}
    = \abs{e(V_1,\dots,V_k)-\frac{r}{n} \prod_{i} \size{V_i}}
    \leq \rho \sqrt{\prod_i \size{V_i}}.
\end{align}

By Theorem \ref{thm:inverse_mixing_hypergraph}, 

\begin{align}
\norm{A_H-\tfrac{r}{n}A_K} = O\parens{(\log^{k-1}(rn^{k-2}/\rho)+1)\rho},
\end{align}





and hence by Proposition \ref{D_bound}

\begin{align}\label{eqn:lambda2bound}
    \lambda_2(H) 
    &= \norm{A_H-\tfrac{r}{n}J}
    \leq \tfrac{r}{n}\norm{D} + \norm{A_H-\tfrac{r}{n}A_K}\\
    &=  O\parens{rn^{(k-4)/2} + (\log^{k-1}(rn^{k-2}/\rho)+1)\rho}.
\end{align}

A similar calculation to the one above also gives a lower bound of
\begin{align}\label{eqn:lambda_2_lower_bound}
    \lambda_2(H) \geq \Omega(rn^{(k-4)/2}) - \rho\, O(\log^{k-1}(rn^{k-2}/\rho) + 1).
\end{align}
\end{proof}

If the first term dominates in \eqref{eqn:lambda_2_lower_bound} then the asymptotics of $\lambda_2$ are independent of $\rho$ and so there is no interesting inverse mixing for this definition of the second eigenvalue.  We now show that this is in fact typically the case by examining $\rho_\alpha$ for random hypergraphs.

To get some idea about the typical magnitude of $\rho_\alpha$, we analyze the Erd\H{o}s-Renyi random hypergraph $G(n, \alpha, k)$, in which each of the $\binom{n}{k}$ $k$-tuples is taken as a hyperedge independently with probability $\alpha$.

\begin{proposition}\label{thm:random_rho}
    For the Erd\H{o}s-Renyi random hypergraph $G = G(n, \alpha, k)$, with high probability $\rho_\alpha(G) = O(\sqrt{n})$.
\end{proposition}
\begin{proof}
    For fixed disjoint sets of vertices $V_1,\dots,V_k$, note that $e(V_1,\dots,V_k)$ is a sum of $\prod_{i=1}^k \size{V_i}$ independent Bernoulli random variables each with mean $\alpha$.

    By Hoeffding's inequality \cite{H}, its deviation from its mean satisfies
    

    \begin{align}
        \Pr\bracks{\Abs[\Big]{e(V_1,\dots,V_k) - \alpha \prod_i \size{V_i}} > t}
        \leq 2e^{-2 t^2/\prod_i \size{V_i}}.
    \end{align}
    Plugging in $t = \rho\sqrt{\prod_i \size{V_i}}$ gives
    \begin{align}
        \Pr\bracks{\Abs[\Big]{e(V_1,\dots,V_k) - \alpha \prod_i \size{V_i}} > \rho\sqrt{\prod_i \size{V_i}}}
        \leq 2e^{-2\rho^2}.
    \end{align}

    Finally, taking a union bound over all $\leq (k+1)^n$ choices of subsets $V_i$, we find that as long as
    \begin{align}
        \delta &\geq 2(k+1)^n e^{-2\rho^2},& \text{or equivalently}&&
        \rho &\geq \sqrt{\frac{n\log (k+1) + \log (2/\delta)}{2}},
    \end{align}
    
    then with probability at least $1-\delta$ the random hypergraph satisfies
    \begin{align}\label{eqn:hypergraph_rho}
        \abs{e(V_1,\dots,V_k)-\alpha e_K(V_1,\dots,V_k)}
        \leq \rho\sqrt{\prod_{i = 1}^k \size{V_i}}.
    \end{align}  
for \emph{all} choices of $V_1,\dots,V_k$.  In particular, for $\delta = e^{-n}$ we have $\rho_\alpha(G) = O(\sqrt{n})$ with probability at least $1-\delta$.

\end{proof}

We can prove a converse in the case where $\alpha$ is constant with respect to $n$.

\begin{proposition} 
For any hypergraph $H$ and any constant $\alpha\in [0,1]$, 
\begin{align*}\rho_{\alpha}(H) \geq \frac{\alpha(1-\alpha)}{\sqrt{\alpha^2 + (1-\alpha^2)}}\sqrt{n-k+1}.
\end{align*}
\end{proposition}

\begin{proof}

Set $V_1\dots, V_{k-1}$ to be distinct singletons $v_1, \dots v_{k-1}$.  Define \begin{align*} S & = \{v\in V-\{v_1,\dots v_{k-1}\}: \{v_1,\dots v_{k-1},v\}\in E(H)\}\\
 \text{and\quad} T & = \{v\in V-\{v_1,\dots v_{k-1}\}: \{v_1,\dots v_{k-1},v\}\notin E(H)\}\end{align*}
 
 Observe that $\size{S}+\size{T} = n-k+1$.  

 Also, that \begin{align*}\rho_{\alpha}(H)&\geq \max \parens{\frac{\Abs[\big]{e(\set{v_1},\dots,\set{v_{k-1}},S)-\alpha\size{S}}}{\sqrt{\size{S}}}, \frac{\Abs[\big]{e(\set{v_1},\dots,\set{v_{k-1}},T)-\alpha\size{T}}}{\sqrt{\size{T}}}}\\
& =\max \parens{(1-\alpha)\sqrt{\size{S}}, \alpha\sqrt{\size{T}}}.
 \end{align*}
 
 If $|S|\geq (n-k+1)\frac{\alpha^2}{(1-\alpha)^2 + \alpha^2}$, then \begin{align*}
 (1-\alpha)\sqrt{|S|}\geq \frac{\alpha(1-\alpha)}{\sqrt{\alpha^2 + (1-\alpha^2)}}\sqrt{n-k+1}.
 \end{align*}
 
 On the other hand, if $|S|\leq (n-k+1)\frac{\alpha^2}{(1-\alpha)^2 + \alpha^2}$, then \begin{align*}|T| = (n-k+1)-|S|\geq (n-k+1)\frac{(1-\alpha)^2}{(1-\alpha)^2 + \alpha^2},\end{align*} and \begin{align*}\alpha\sqrt{|T|}\geq \frac{\alpha(1-\alpha)}{\sqrt{\alpha^2 + (1-\alpha^2)}}\sqrt{n-k+1}.\end{align*}

\end{proof}


Combining this with the previous proposition proves

\begin{proposition} \label{prop:randomrho}
    For the dense Erd\H{o}s-Renyi random hypergraph $G = G(n,\alpha,k)$ where $\alpha \in (0,1)$ is constant with respect to $n$, with high probability $\rho_{\alpha}(G) = \Theta(\sqrt{n})$.
\end{proposition}

Assume that $\alpha = \tfrac{r}{n}$ is a positive constant independent of $n$, in other words, that $r = \Theta(n)$.

\begin{corollary}
For the dense Erd\H{o}s-Renyi random hypergraph $G = G(n,\alpha,k)$ where $\alpha \in (0,1)$ is constant with respect to $n$ and $k\geq 4$,with high probability $\lambda_2 = \Theta(n^{(k-2)/2})$.
\end{corollary}

For $G$ that satisfies the bound in Proposition \ref{prop:randomrho}, the second term of \eqref{eqn:lambda_2_lower_bound} is $\Theta(\sqrt{n}\log^{k-1}(n))$, which is dominated by the first term if $k\geq 4$.  So, almost every hypergraph with $k \geq 4$ will not have an interesting inverse mixing lemma for the Friedman-Wigderson definition of the second eigenvalue. 

\begin{corollary}
For the dense Erd\H{o}s-Renyi random hypergraph $G = G(n,\alpha,k)$ where $\alpha \in (0,1)$ is constant with respect to $n$, with high probability $\lambda_{2,\alpha}(G) = \Omega(\sqrt{n})$ and $\lambda_{2,\alpha}(G) = O(\sqrt{n}\log^{k-1}{n})$.
\end{corollary}

This is proven by combining Proposition \ref{prop:randomrho} with the bounds on $\lambda_{2,\alpha}$ found in Theorem \ref{thm:mixing_hypergraph_density} and Theorem \ref{thm:inverse_mixing_hypergraph}.

\nocite{Ch70}
\nocite{P13}
\nocite{SKM12}
\bibliography{main}
\bibliographystyle{plainnat}

\end{document}